\documentclass[reqno,11pt]{amsart}
\usepackage{amssymb,amsmath,amsthm,mathtools}
\usepackage{enumerate}
\usepackage[letterpaper,margin=1in]{geometry}

\def\Z{\mathbb{Z}}
\def\Q{\mathbb{Q}}
\def\R{\mathbb{R}}
\def\H{\mathbb{H}}
\def\N{\mathbb{N}}
\def\C{\mathbb{C}}
\def\P{\mathbb{P}}

\def\PSL{{\rm PSL}}

\def\sQ{\mathcal{Q}}
\def\reg{{\rm reg}}
\newcommand{\pfrac}[2]{\left(\frac{#1}{#2}\right)}
\newcommand{\ptfrac}[2]{\left(\tfrac{#1}{#2}\right)}
\newcommand{\pMatrix}[4]{\left(\begin{matrix}#1 & #2 \\ #3 & #4\end{matrix}\right)}
\renewcommand{\pmatrix}[4]{\left(\begin{smallmatrix}#1 & #2 \\ #3 & #4\end{smallmatrix}\right)}
\renewcommand{\bar}[1]{\overline{#1}}

\DeclareMathOperator{\M}{\mathbb{M}}
\DeclareMathOperator{\sgn}{sgn}
\DeclareMathOperator{\re}{Re}
\DeclareMathOperator{\im}{Im}
\DeclareMathOperator{\Tr}{Tr}

\newtheorem{theorem}{Theorem}
\newtheorem{lemma}[theorem]{Lemma}
\newtheorem{proposition}[theorem]{Proposition}
\theoremstyle{remark}
\newtheorem*{remark}{Remark}

\numberwithin{equation}{section}

\begin{document}

\title[Periods of the $j$-function along infinite geodesics]{Periods of the $j$-function along infinite geodesics and mock modular forms}
\author{Nickolas Andersen}
\address{Department of Mathematics\\
University of Illinois\\
Urbana, IL 61801} 
\email{nandrsn4@illinois.edu}

\begin{abstract}
Zagier's well-known work on traces of singular moduli relates the coefficients of certain weakly holomorphic modular forms of weight $1/2$ to traces of values of the modular $j$-function at imaginary quadratic points. A real quadratic analogue was recently studied by Duke, Imamo\=glu, and T\'oth. They showed that the coefficients of certain weight $1/2$ mock modular forms
\[
	f_D = \sum_{d>0} a(d,D) q^d, \qquad D>0
\] 
are given in terms of traces of cycle integrals of the $j$-function. 
Their result applies to those coefficients $a(d,D)$ for which $dD$ is not a square. 
Recently Bruinier, Funke, and Imamo\=glu employed a regularized theta lift to show that the coefficients $a(d,D)$ for square $dD$ are traces of regularized integrals of the $j$-function.
In the present paper we provide an alternate approach to this problem. 
We introduce functions $j_{m,Q}$ (for $Q$ a quadratic form) which are related to the $j$-function and show, by modifying the method of Duke, Imamo\=glu, and T\'oth, that the coefficients for which $dD$ is a square are traces of cycle integrals of the functions $j_{m,Q}$.
\end{abstract}

\maketitle

\section{Introduction}
For a nonzero integer $d\equiv 0,1\pmod{4}$, let $\sQ_d$ denote the set of binary quadratic forms $Q(x,y)=[a,b,c]=ax^2+bxy+cy^2$ with discriminant $b^2-4ac=d$ which are positive definite if $d<0$. 
The modular group $\Gamma=\PSL_2(\Z)$ acts on these forms in the usual way, resulting in finitely many classes $\Gamma \backslash \sQ_d$.

If $d<0$ and $Q\in \sQ_d$ then $Q(x,1)$ has exactly one root $\tau_Q$ in $\H$, namely
\[
	\tau_Q = \frac{-b+\sqrt{d}}{2a}.
\]
The values of the modular $j$-invariant
\[
	j(\tau) := \frac1q + 744 + 196884q + \cdots, \qquad q:=e^{2\pi i \tau}
\]
at the points $\tau_Q$ are called \emph{singular moduli}; they are algebraic integers which play many important roles in number theory. 
For instance, when $d$ is a fundamental discriminant (i.e. the discriminant of $\Q(\sqrt{d})$), the field $\Q(j(\tau_Q))$ is the Hilbert class field of $\Q(\tau_Q)$. 

For $Q\in \sQ_d$, let $\Gamma_Q$ denote the stabilizer of $Q$ in $\Gamma$. 
Then $\Gamma_Q=\{1\}$ unless $Q\sim [a,0,a]$ or $Q\sim [a,a,a]$, in which case it has order $2$ or $3$, respectively. 
For $f\in \C[j]$, we define the modular trace of $f$ by
\begin{equation} \label{eq:def-tr-neg-d}
	\Tr_d(f) := \sum_{Q\in \Gamma \backslash \sQ_d} \frac{1}{|\Gamma_Q|} f(\tau_Q).
\end{equation}
A well-known theorem of Zagier \cite{Zagier:Traces} states that, for $j_1 := j-744$, the series
\[
	g_{1}(\tau) := \frac1q - 2 - \sum_{0>d\equiv 0,1(4)} \Tr_{d}(j_1) \, q^{-d}
\]
is in $M_{3/2}^!$, the space of weakly holomorphic modular forms of weight $3/2$ on $\Gamma_0(4)$ which satisfy the plus space condition (see Section \ref{sec:poincare} for details). 
Zagier further showed that $g_1$ is the first member of a basis $\{g_D\}_{0<D\equiv 0,1(4)}$ for $M_{3/2}^!$. 
Each function $g_D$ is uniquely determined by having a Fourier expansion of the form
\begin{equation} \label{eq:g-D-basis}
	g_D(\tau) = q^{-D} - \sum_{0>d\equiv 0,1(4)} a(D,d) q^{-d}.
\end{equation}
The coefficients $a(D,d)$ with $D$ a fundamental discriminant are given by
\[
	a(D,d) = -\Tr_{d,D}(j_1),
\]
where $\Tr_{d,D}$ denotes the twisted trace
\begin{equation} \label{eq:twisted-traces}
	\Tr_{d,D}(f) := \frac{1}{\sqrt{D}}\sum_{Q\in \Gamma \backslash \sQ_{dD}} \frac{\chi_D(Q)}{|\Gamma_Q|} f(\tau_Q),
\end{equation}
and $\chi_D:\sQ_{dD}\to\{\pm 1\}$ is defined in \eqref{eq:def-chi-D} below.

If $Q$ has positive nonsquare discriminant, then $Q(x,1)$ has two irrational roots. 
Let $S_Q$ denote the geodesic in $\H$ connecting the roots, oriented counter-clockwise if $a>0$ and clockwise if $a<0$. 
In this case the stabilizer $\Gamma_Q$ is infinite cyclic, and $C_Q:=\Gamma_Q\backslash S_Q$ defines a closed geodesic on the modular curve. 
In analogy with \eqref{eq:twisted-traces} we define, for $dD>0$ not a square,
\begin{equation} \label{eq:def-tr-pos-nonsq-d}
	\Tr_{d,D}(f) := \frac{1}{2\pi}\sum_{Q\in \Gamma \backslash \sQ_{dD}} \chi_D(Q) \int_{C_Q} f(\tau) \frac{d\tau}{Q(\tau,1)}.
\end{equation}
Let $\M_{1/2}^+$ denote the space of mock modular forms of weight $1/2$ on $\Gamma_0(4)$ satisfying the plus space condition (see Section \ref{sec:poincare} for definitions). 
A beautiful result of Duke, Imamo\=glu, and T\'oth \cite{DIT:CycleIntegrals} shows that the twisted traces \eqref{eq:twisted-traces} and \eqref{eq:def-tr-pos-nonsq-d} appear as coefficients of mock modular forms in a basis $\{f_D\}_{D\equiv 0,1(4)}$ for $\M_{1/2}^+$. 
When $D<0$, the form $f_D$ is a weakly holomorphic modular form, and is uniquely determined by having a Fourier expansion of the form
\[
	f_D(\tau) = q^D + \sum_{0<d\equiv 0,1(4)} a(d,D) q^d.
\]
The coefficients $a(d,D)$ are the same as those in \eqref{eq:g-D-basis}. 
Therefore, when $D$ is a fundamental discriminant, they are given in terms of twisted traces. 
When $D>0$ the mock modular form $f_D$ is uniquely determined by being holomorphic at $\infty$ and having shadow equal to $2g_D$ (see Section \ref{sec:poincare}). 
Let
\[
	f_D(\tau) = \sum_{0<d\equiv 0,1(4)} a(d,D) \, q^d.
\]
If $D$ is a fundamental discriminant and $dD$ is not a square, then Theorem 3 of \cite{DIT:CycleIntegrals} shows that
\[
	 a(d,D) = \Tr_{d,D}(j_1).
\]
In \cite{DIT:CycleIntegrals} the coefficients $a(d,D)$ for square $dD$ are defined as infinite series involving Kloosterman sums and the $J$-Bessel function. The authors leave an arithmetic or geometric interpretation of these coefficients as an open problem. 

When the discriminant of $Q$ is a square, the stabilizer $\Gamma_Q$ is trivial. 
In this case the geodesic $C_Q$ connects two elements of $\P^1(\Q)$, but since any $f\in \C[j]$ has a pole at $\infty$ (which is $\Gamma$-equivalent to every element of $\P^1(\Q)$), the integral
\begin{equation} \label{eq:div-int}
	\int_{C_Q} f(\tau) \, \frac{d\tau}{Q(\tau,1)}
\end{equation}
diverges. 
This is the obstruction to a geometric interpretation of the modular trace for square discriminants.
In a recent paper, Bruinier, Funke, and Imamo\=glu \cite{BFI:RegularizedTheta} address this issue by regularizing the integral \eqref{eq:div-int} and showing that the corresponding modular traces
\[
	\Tr_d(j_1) = \frac{1}{2\pi}\sum_{Q\in \Gamma \backslash \sQ_d} \int_{C_Q}^{\reg} j_1(\tau) \frac{d\tau}{Q(\tau,1)}
\]
give the coefficients of $f_1$.
Their proof is quite different than the argument given in \cite{DIT:CycleIntegrals} for nonsquare discriminants.
It involves a regularized theta lift and applies to a much more general class of modular functions (specifically, weak harmonic Maass forms of weight $0$ on any congruence subgroup of $\Gamma$).

In this paper we provide an alternate definition of $\Tr_{d,D}$ when $dD$ is a square which does not rely on regularizing a divergent integral. 
Instead, we show that the coefficients of $f_D$ for square $dD$ are given in terms of convergent integrals of functions $j_{1,Q}$ which are related to $j_1$. 
Furthermore, using this definition we show that a suitable modification of the proof of Theorem 3 of \cite{DIT:CycleIntegrals} for nonsquare discriminants works for all discriminants.

We first define a sequence of modular functions $\{j_m\}_{m\geq 0}$ which forms a basis for the space $\C[j]$. 
We let $j_0:=1$ and for $m\geq 1$ we define $j_m$ to be the unique modular function of the form
\[
	j_m(\tau) = q^{-m} + \sum_{n>0} c_m(n) q^n.
\]
Note that $j_1=j-744$ was already defined above.

We define the functions $j_{m,Q}$ as follows. 
When the discriminant of $Q$ is a square, each root of $Q(x,y)$ corresponds to a cusp $\alpha=\frac rs \in \P^1(\Q)$ with $(r,s)=1$. 
Let $\gamma_\alpha:=\pmatrix **s{-r} \in \Gamma$ be a matrix that sends $\alpha$ to $\infty$, and define
\[
	j_{m,Q}(\tau) := 
	j_m(\tau) - 2\sum_{\alpha \in \{\text{roots of }Q\}}
		\sinh(2\pi m \im \gamma_\alpha \tau) \, e(m\re \gamma_\alpha \tau),
\]
where $e(x) := e^{2\pi i x}$. 
Note that there are only two terms in the sum. 
When $dD>0$ is a square, we define the twisted trace of $j_m$ by
\begin{equation} \label{eq:def-tr-sq-d}
	\Tr_{d,D}(j_m) := \frac{1}{2\pi}\sum_{Q\in \Gamma \backslash \sQ_{dD}} \chi_D(Q) \int_{C_Q} j_{m,Q}(\tau) \frac{d\tau}{Q(\tau,1)}.
\end{equation}
\begin{remark}
If $\alpha$ is a root of $Q$ and $\sigma\in \Gamma$, then $\sigma \alpha$ is a root of $\sigma Q$ (see \eqref{eq:action-Gamma} below). 
Since $\gamma_{\sigma\alpha} \sigma = \gamma_\alpha$, we have $j_{m,\sigma Q}(\sigma \tau)=j_{m,Q}(\tau)$.
Together with \eqref{eq:d-tau} below and the fact that $\chi_D(\sigma Q) = \chi_D(Q)$, this shows that the summands in \eqref{eq:def-tr-sq-d} remain unchanged by $Q\mapsto \sigma Q$. Therefore $\Tr_{d,D}(j_m)$ is well-defined.
\end{remark}

\begin{theorem} \label{thm:main}
Suppose that $0<d\equiv 0,1\pmod{4}$ and that $D>0$ is a fundamental discriminant.
With $\Tr_{d,D}(j_1)$ defined in \eqref{eq:def-tr-pos-nonsq-d} and \eqref{eq:def-tr-sq-d} for nonsquare and square $dD$, respectively, the function
\[
	f_D(\tau) = \sum_{0<d\equiv 0,1(4)} \Tr_{d,D}(j_1) \, q^d
\]
is a mock modular form of weight $1/2$ for $\Gamma_0(4)$ with shadow $2g_D$.
\end{theorem}

It is instructive to consider the special case $d=D=1$. 
In this case, there is one quadratic form $Q=[0,1,0]$ with roots $0$ and $\infty$, so $C_Q$ is the upper half of the imaginary axis. 
Then
\[
	j_{m,Q}(iy) = j_m(iy) - 2\sinh(2\pi m y) - 2\sinh(2\pi m/y),
\]
and we have
\[
	\lim_{y\to 0^+} \frac{j_{m,Q}(iy)}{y} = -4\pi m.
\]
Since $j_{m,Q}(iy)/y = O(1/y^2)$ as $y\to \infty$, the integral
\begin{equation} \label{eq:tr1-j1}
	\Tr_{1,1}(j_m) = \frac{1}{2\pi}\int_0^\infty j_{m,Q}(iy) \frac{dy}{y}
\end{equation}
converges. 
Theorem \ref{thm:main} shows that $\Tr_{1,1}(j_1)=-16.028\ldots$ is the coefficient of $q$ in the mock modular form~$f_1$.

\begin{remark}
The regularization in \cite[eq. (1.10)]{BFI:RegularizedTheta} of the integral \eqref{eq:div-int} essentially amounts to replacing the divergent integral
\[
	\int_1^\infty e^{2\pi y} \, \frac{dy}{y} 
	= \int_{-2\pi}^{-\infty} e^{-t} \, \frac{dt}{t}
\]
by $-107.47\ldots$, which is the Cauchy principal value of the integral
\[
	\int_{-2\pi}^\infty e^{-t} \frac{dt}{t}.
\]
If these were equal, we could deduce that
\[
	\int_0^\infty \left(2\sinh(2\pi y) + 2\sinh(2\pi/y)\right) \frac{dy}{y} = 0,
\]
so the values of $\Tr_{1,1}(j_1)$ in \cite{BFI:RegularizedTheta} and \eqref{eq:tr1-j1} agree.
\end{remark}

The modular traces $\Tr_{d,D}(j_m)$ for $m>1$ are also related to the coefficients $a(D,d)$. With the modular trace now defined when $dD$ is a square, we obtain Theorem 3 of \cite{DIT:CycleIntegrals} with the condition ``$dD$ not a square'' removed. 
Theorem \ref{thm:main} follows as a corollary.

\begin{theorem} \label{thm:divisor-sum}
Let $a(D,d)$ be the coefficients defined above. For $0<d\equiv 0,1\pmod{4}$ and $D>0$ a fundamental discriminant we have
\begin{equation} \label{eq:thm-2}
	\Tr_{d,D}(j_m) = \sum_{n|m} \pfrac{D}{m/n} n \, a(n^2 D,d).
\end{equation}
\end{theorem}

In Section \ref{sec:preliminaries} we recall some facts about binary quadratic forms, focusing on forms of square discriminant. 
In Section \ref{sec:poincare} we define mock modular forms and describe the functions $j_{m,Q}$ in terms of Poincar\'e series.
The proof of Theorem \ref{thm:divisor-sum} comprises Section \ref{sec:proofs}. 
We follow the proof given in \cite{DIT:CycleIntegrals} for nonsquare discriminants, modifying as needed when the discriminant is a square.

\section{Binary quadratic forms} \label{sec:preliminaries}
In this section, we recall some basic facts about binary quadratic forms and the characters $\chi_D$, and we give an explicit description of the classes $\Gamma\backslash \sQ_{d}$ when $d>0$ is a square. 
Throughout, we assume that $d,D\equiv 0,1\pmod{4}$.

Recall that the left action of $\gamma=\pmatrix ABCD \in \Gamma$ on $Q(x,y)$ is given by the right action of $\gamma^{-1}$; that is,
\begin{equation} \label{eq:action-Gamma}
	\gamma Q = Q\gamma^{-1} = Q(Dx-By,-Cx+Ay).
\end{equation}
This action is compatible with the linear fractional action $\gamma\tau=\frac{A\tau+B}{C\tau+D}$ on the roots of $Q(\tau,1)$; if $\tau_Q$ is a root of $Q$, then $\gamma\tau_Q$ is a root of $\gamma Q$.

Suppose that $D$ is a fundamental discriminant. If $Q=[a,b,c] \in \sQ_{dD}$, we define
\begin{equation} \label{eq:def-chi-D}
	\chi_D(Q) := 
	\begin{cases}
		\pfrac{D}{r} &\text{ if $(a,b,c,D)=1$ and $Q$ represents $r$ with $(r,D)=1$},\\
		0 &\text{ if $(a,b,c,D)>1$}.
	\end{cases}
\end{equation}
The basic theory of these characters is presented nicely in \cite[Section 2]{GKZ:Heegner}. It turns out that $\chi_D$ is well-defined on classes $\Gamma \backslash \sQ_{dD}$ and that
\[
	\chi_D(-Q) = (\sgn D)\chi_D(Q).
\]

If $Q=[a,b,c]\in \sQ_d$ with $d>0$ then the cycle $S_Q$ is the curve in $\H$ defined by the equation
\[
	a|\tau|^2 + b\re \tau + c=0.
\]
When $a=0$, $S_Q$ is the vertical line $\re \tau=-c/b$ oriented upward. When $a\neq 0$, $S_Q$ is a semicircle oriented counterclockwise if $a>0$ and clockwise if $a<0$. If $\gamma\in \Gamma$ then we have $\gamma S_Q = S_{\gamma Q}$. We define
\[
	d\tau_Q := \frac{\sqrt{d} \, d\tau}{Q(\tau,1)},
\]
so that if $\tau'=\gamma\tau$ for some $\gamma\in \Gamma$, we have 
\begin{equation} \label{eq:d-tau}
	d\tau'_{\gamma Q} = d\tau_Q.
\end{equation}

When $d>0$ is a square, we can describe a set of representatives for $\Gamma \backslash \sQ_d$ explicitly, as the next lemma shows.

\begin{lemma} \label{lem:square-reps}
Suppose that $d=b^2$ for some $b\in \N$. Then the set
\[
	\left\{ [a,b,0] : 0\leq a<b \right\}
\]
is a complete set of representatives for $\Gamma \backslash \sQ_d$.
\end{lemma}

\begin{proof}
Let $Q\in \sQ_d$. We will show that
\begin{enumerate}
	\item $Q\sim [a,b,0]$ for some $a$ with $0\leq a<b$, and
	\item if $[a,b,0]\sim [a',b,0]$ then $a\equiv a'\pmod{b}$.
\end{enumerate}
Since the roots of $Q(x,y)$ are rational, there exist integers $r,s,t,u$ with $(r,s)=1$ such that
\[
	Q(x,y) = (rx+sy)(tx+uy).
\]
If $\gamma = \pmatrix rs** \in \Gamma$ then $\gamma Q=[a,\varepsilon b,0]$ for some $\varepsilon \in \{\pm 1\}$ and some $a\in \Z$. Since $\pmatrix 10k1 [a,\varepsilon b,0]=[a-\varepsilon k b,\varepsilon b,0]$ we may assume that $0\leq a<b$.
Suppose that $\varepsilon=-1$. 
Let $g=(a,b)$ and define $\bar{a}$ by the conditions $a\bar{a}\equiv g^2\pmod b$ and $0\leq \bar{a}<b$. 
Then
\[
	\pMatrix{a/g}{-b/g}{*}{\bar{a}/g}[a,-b,0]=[\bar{a},b,0],
\]
and claim ($1$) follows.

Suppose that $[a,b,0]\sim [a',b,0]$. Then there exists $\pmatrix ABCD \in \Gamma$ with $A>0$ such that
\begin{gather}
	D(aD-bC)=a', \label{eq:sq-red-1} \\
	b(AD+BC)-2aBD=b, \label{eq:sq-red-2} \\
	B(aB-Ab)=0. \label{eq:sq-red-3}
\end{gather}
Let $g=(a,b)$. If $aB-Ab=0$ then $A=a/g$ and $B=b/g$, so \eqref{eq:sq-red-2} implies that $AD-BC=-1$, a contradiction. So by \eqref{eq:sq-red-3} we have $B=0$ which, together with \eqref{eq:sq-red-2}, implies that $AD=1$. Then \eqref{eq:sq-red-1} shows that $a'\equiv aD^2\equiv a\pmod b$. This proves claim ($2$).
\end{proof}

\section{Mock modular forms and Poincar\'e series} \label{sec:poincare}

We define mock modular forms following \cite{DIT:CycleIntegrals} (see also \cite{BF:ThetaLifts}, \cite{Ono:Visions}, and \cite{Zagier:MockTheta}).
Let $k\in 1/2+\Z$. We say that $f:\H\to\C$ has weight $k$ for $\Gamma_0(4)$ if for all $\pmatrix abcd \in \Gamma_0(4)$ we have
\begin{equation} \label{eq:transform}
f\left(\frac{a\tau+b}{c\tau+d}\right) = \pfrac cd^{2k} \varepsilon_d^{-2k} (c\tau+d)^k f(\tau),
\end{equation}
where $\pfrac cd$ is the Kronecker symbol and
\[
\varepsilon_d := 
\begin{cases}
	1 &\text{ if } d\equiv 1\pmod4,\\
	i &\text{ if } d\equiv 3\pmod4.
\end{cases}
\]
We say that $f=\sum a(n) q^n$ satisfies the plus space condition if the coefficients $a(n)$ are supported on integers $n \gg -\infty$ with $(-1)^{k-1/2}n\equiv 0,1\bmod 4$.
Let $M_k^!$ denote the space of functions which are holomorphic on $\H$, have weight $k$ for $\Gamma_0(4)$, and satisfy the plus space condition.

A holomorphic function $f:\H\to\C$ which satisfies the plus space condition is called a mock modular form of weight $1/2$ if there exists a function $g\in M_{3/2}^!$, called the shadow of $f$, such that the completed function $f+g^*$ has weight $1/2$ for $\Gamma_0(4)$. 
Here $g^*$ is the nonholomorphic Eichler integral defined in (1.4) of \cite{DIT:CycleIntegrals}.

In Section 2 of \cite{DIT:CycleIntegrals}, the mock modular forms $f_D$ are constructed explicitly using nonholomorphic Maass-Poincar\'e series. 
For $D>0$ the form $f_D$ is the holomorphic part of $D^{-1/2}h_D$, where $h_D$ is defined in Proposition 1 of \cite{DIT:CycleIntegrals}. If
\[
	f_D(\tau) = \sum_{0<d\equiv 0,1(4)} a(d,D) q^d
\]
then by (2.15), (2.21), (2.29), and Lemma 5 of \cite{DIT:CycleIntegrals} we have
\begin{equation} \label{eq:a-d-D-limit}
	a(d,D) = (dD)^{-\frac12}\lim_{s\to\frac34^+} \left( b(d,D,s) - \frac{b(d,0,s)b(0,D,s)}{b(0,0,s)} \right),
\end{equation}
where
\begin{equation} \label{eq:b-d-D-s-formula}
	b(d,D,s) = \sum_{c=1}^\infty K^+(d,D;4c) \times
	\begin{cases}
		2^{-\frac32}\pi (dD)^{\frac 14}c^{-1}J_{2s-1}\pfrac{\pi \sqrt{dD}}{c} &\text{if } dD>0,\\
		2^{-4s}\pi^{s+\frac 14} (d+D)^{s-\frac 14}c^{-2s} &\text{if $dD=0$ and $d+D\neq 0$},\\
		2^{\frac12-6s} \pi^\frac12 \Gamma(2s) c^{-2s} &\text{if }d=D=0.
	\end{cases}
\end{equation}
Here $J_{2s-1}$ is the $J$-Bessel function and $K^+(d,D;4c)$ is the modified Kloosterman sum
\[
	K^+(d,D;4c) := (1-i) \sum_{a\bmod 4c} \pfrac {4c}a \varepsilon_a e\pfrac{da+D\bar a}{4c} \times 
	\begin{cases}
		1 &\text{ if $c$ is even},\\
		2 &\text{ otherwise},
	\end{cases}
\]
where $\bar a$ denotes the inverse of $a$ modulo $4c$. 
Equation \eqref{eq:b-d-D-s-formula} shows that $b(d,D,s)=b(D,d,s)$, so for $d,D>0$ we have
\begin{equation} \label{eq:d-D-symmetric}
	a(d,D) = a(D,d).
\end{equation}

To prove Theorem \ref{thm:divisor-sum} we need to express $j_{m,Q}(\tau,s)$ in terms of certain modified Poincar\'e series $G_{m,Q}(\tau,s)$.
Let $\phi:\R^+\to \C$ be a smooth function satisfying $\phi(y) = O_\epsilon(y^{1+\epsilon})$ for any $\epsilon>0$, and let $m\in \Z$. 
Define the Poincar\'e series associated to $\phi$ by
\begin{equation} \label{eq:def-G-m}
	G_m(\tau,\phi) := \sum_{\gamma \in \Gamma_\infty \backslash \Gamma} e(-m \re \gamma \tau) \phi(\im \gamma \tau).
\end{equation}
As in \cite{Fay} and \cite{Niebur}, we make the specialization
\begin{equation} \label{eq:def-phi}
	\phi(y) = \phi_{m,s}(y) :=
	\begin{cases}
		y^s &\text{ if }m=0,\\
		2\pi |m|^\frac12 y^\frac12 I_{s-\frac12}(2\pi|m|y) &\text{ if }m\neq 0,
	\end{cases}
\end{equation}
where $I_{s-\frac12}$ is the $I$-Bessel function and $\re s>1$ (to guarantee convergence). We write $G_m(\tau,s):=G_m(\tau,\phi_{m,s})$ and we define
\begin{equation} \label{eq:def-j-m}
	j_m(\tau,s) := G_m(\tau,s) - \frac{2 \pi^{s+\frac12} m^{1-s}\sigma_{2s-1}(m)}{\Gamma(s+\frac12)\zeta(2s-1)}G_0(\tau,s).
\end{equation}
As explained in Section 4 of \cite{DIT:CycleIntegrals} and Section 6.4 of \cite{BFI:RegularizedTheta}, when $m>0$ the function $G_m(\tau,s)$ has an analytic continuation to $\re s>3/4$, and when $m=0$ the function $G_m(\tau,s)$ has a pole at $s=1$ arising from its constant term. 
The factor multiplied by $G_0(\tau,s)$ in \eqref{eq:def-j-m} is chosen to cancel the pole of $G_0(\tau,s)$ at $s=1$ and to eliminate the constant term of $G_m(\tau,1)$.
Furthermore, we have $j_m(\tau,1)=j_m(\tau)$.

Recall that for $d>0$ a square and $Q\in \sQ_d$, the functions $j_{m,Q}(\tau)$ are defined as
\begin{equation} \label{eq:def-j-m-Q-2}
	j_{m,Q}(\tau) := 
	j_m(\tau) - 2\sum_{\alpha \in \{\text{roots of }Q\}}
		\sinh(2\pi m \im \gamma_\alpha \tau) \, e(m\re \gamma_\alpha \tau).
\end{equation}
Since $\phi_{m,1}(y) = 2\sinh(2\pi |m| y)$, the two terms subtracted from $j_m(\tau)$ in \eqref{eq:def-j-m-Q-2} are the terms in the Poincar\'e series \eqref{eq:def-G-m} corresponding to $\gamma_\alpha$ for the roots $\alpha$ of $Q$. 
It turns out that these are the terms which cause the integral
\[
	\int_{C_Q} G_m(\tau,1) \frac{d\tau}{Q(\tau,1)}
\]
to diverge. 
In analogy with \eqref{eq:def-j-m} and \eqref{eq:def-j-m-Q-2}, we define
\begin{equation} \label{def-j-m-Q-tau-s}
	j_{m,Q}(\tau,s) := G_{m,Q}(\tau,s) - \frac{2 \pi^{s+\frac12} m^{1-s}\sigma_{2s-1}(m)}{\Gamma(s+\frac12)\zeta(2s-1)}G_{0,Q}(\tau,s),
\end{equation}
where $G_{m,Q}(\tau,s)$ is the modified Poincar\'e series
\[
	G_{m,Q}(\tau,s) := 
	\sum_{\substack{\gamma \in \Gamma_\infty \backslash \Gamma \\ \gamma \, \neq \, \gamma_\alpha}}
	e(-m \re \gamma \tau) \phi_{m,s}(\im \gamma \tau).
\]
Since the two terms subtracted from $G_0(\tau,s)$ are killed by the pole of $\zeta(2s-1)$, we conclude that
\begin{equation}
	j_{m,Q}(\tau,1) = j_{m,Q}(\tau).
\end{equation}
Therefore, to compute the cycle integrals of the functions $j_{m,Q}(\tau)$, it is enough to compute the cycle integrals of the functions $G_{m,Q}(\tau,s)$.

\section{Proof of Theorem \ref{thm:divisor-sum}} \label{sec:proofs}
Throughout this section we assume that $dD>0$ is a square. 
The main ingredient in the proof of Theorem \ref{thm:divisor-sum} is the following proposition, which computes the traces of the functions $G_{m,Q}(\tau,s)$ in terms of the $J$-Bessel function and the exponential sum
\[
	S_m(d,D;4c) := \sum_{\substack{b\bmod 4c \\ b^2\equiv dD \bmod{4c}}} \chi_D \left([c,b,\tfrac{b^2-dD}{4c}] \right) e\pfrac{mb}{2c}.
\]
See Proposition 4 of \cite{DIT:CycleIntegrals} for the analogous formula for the traces of the functions $G_m(\tau,s)$.

\begin{proposition} \label{prop:tr-g-m}
Let $\re s>1$ and $m \geq 0$. Suppose that $dD>0$ is a square. Then
\[
	\sum_{Q\in \Gamma \backslash \sQ_{dD}} \frac{\chi_D(Q)}{B(s)} \int_{C_Q} G_{m,Q}(\tau,s) \, d\tau_Q = 
	\begin{dcases}
		\tfrac{\pi}{\sqrt 2} m^{\frac 12} (dD)^{\frac 14} \sum_{c=1}^\infty \frac{S_m(d,D;4c)}{c^{\frac 12}} J_{s-\frac 12} \left( \frac{\pi m\sqrt{dD}}{c} \right) &\text{ if }m>0,\\
		2^{-s-1} (dD)^{\frac s2} \sum_{c=1}^\infty \frac{S_0(d,D;4c)}{c^s} & \text{ if } m=0,
	\end{dcases}
\]
where B(s) := $2^s \Gamma(\frac s2)^2/\Gamma(s)$.
\end{proposition}

\begin{proof}
Let $b=\sqrt{dD}$. By Lemma \ref{lem:square-reps}, a complete set of representatives for $\Gamma \backslash \sQ_{dD}$ is given by 
\[
	\left\{Q_a=[a,b,0] : 0\leq a<b\right\}.
\]
Let $g=(a,b)$. 
Then the roots of $Q_a=ax^2+bxy$ in $\P^1(\Q)$ are $0$ and $\beta:=-\frac {b'}{a'}$, where $a'=a/g$ and $b'=b/g$. 
The corresponding matrices are
\[
	\gamma_0 = \pMatrix{0}{-1}{1}{0}, \qquad \gamma_\beta = \pMatrix**{a'}{b'}.
\]
Thus, replacing $\tau$ by $\gamma^{-1}\tau$ in the integral, we have
\[
	\sum_{Q\in \Gamma \backslash \sQ_{dD}} \chi_D(Q) \int_{C_Q} G_{m,Q}(\tau,s) \, d\tau_Q = \sum_{a \bmod b} \chi_D([a,b,0]) \sum_{\substack{\gamma \in \Gamma_\infty \backslash \Gamma \\ \gamma \neq \, \gamma_0,\,\gamma_\beta}} \int_{C_{\gamma Q}} e(-m \re \tau)\phi_{m,s}(\im \tau) \, d\tau_{\gamma Q}.
\]
The map $(\gamma,Q) \mapsto \gamma Q$ is a bijection
\[
	\Gamma_\infty \backslash \Gamma \times \Gamma \backslash \sQ_{dD} \longleftrightarrow \Gamma_\infty \backslash \sQ_{dD}
\]
which sends $(\gamma_0,[a,b,0])$ to $[0,-b,a]$ and $(\gamma_\beta,[a,b,0])$ to $[0,b,g \, \bar{a'}]$, where $a'\bar{a'}\equiv 1\bmod b'$ and $g=(a,b)$. 
Since $\pmatrix 1k01 [0,b,c]=[0,b,c-kb]$, we conclude that
\[
	\sum_{Q\in \Gamma \backslash \sQ_{dD}} \chi_D(Q) \int_{C_Q} G_{m,Q}(\tau,s) \, d\tau_Q 
	= \sum_{\substack{Q\in \Gamma_\infty \backslash \sQ_{dD} \\ Q\neq [0,\pm b,*]}} \chi_D(Q) \int_{C_Q} e(-m \re \tau) \phi_{m,s}(\im \tau) \, d\tau_Q.
\]
The remainder of the proof follows the proofs of Lemmas 7 and 8 and Proposition 4 of \cite{DIT:CycleIntegrals}. 

Since we have eliminated those terms in the sum with $a=0$, we can parametrize each cycle $C_Q$ with $Q=[a,b,c]$ by
\[
	\tau =
	\begin{cases}
		\re \tau_Q + e^{i\theta} \im \tau_Q &\text{ if } a>0, \\
		\re \tau_Q - e^{-i\theta} \im \tau_Q &\text{ if } a<0,
	\end{cases}
	\qquad 0\leq \theta\leq\pi
\]
where
\[
	\tau_Q := -\frac{b}{2a} + i \frac{\sqrt{dD}}{2|a|}
\]
is the apex of the semicircle. 
We then have
\[
	Q(\tau,1) = \frac{dD}{4a}
	\begin{cases}
		e^{2i\theta}-1 &\text{ if }a>0,\\
		e^{-2i\theta}-1 &\text{ if }a<0,
	\end{cases}
\]
which gives $d\tau_Q = d\theta/\sin\theta$. 
Hence for $a\neq 0$ we have
\begin{equation} \label{eq:param-cycle}
	\int_{C_Q} e(-m \re \tau) \phi_{m,s}(\im \tau) \, d\tau_Q = e\pfrac{mb}{2a}\int_0^\pi e\left(-\frac{m\sqrt{dD}}{2a} \cos\theta\right) \phi_{m,s}\left(\frac{\sqrt{dD}}{2|a|} \sin\theta\right) \frac{d\theta}{\sin\theta}.
\end{equation}
Consider the sum of the terms corresponding to $Q$ and $-Q$, where $Q=[a,b,c]$ and $a>0$. Since $\chi_D(Q)=\chi_D(-Q)$
we find that
\begin{multline} \label{eq:Q-minus-Q}
	\chi_D(Q)\int_{C_Q} G_{m,Q}(\tau,s) \, d\tau_Q + \chi_D(-Q)\int_{C_{-Q}} G_{m,-Q}(\tau,s) \, d\tau_{-Q} \\ 
	= 2 \,\chi_D(Q) \,e\pfrac{mb}{2a}\int_0^\pi \cos\left(\frac{\pi m\sqrt{dD}}{a} \cos\theta\right) \phi_{m,s}\left(\frac{\sqrt{dD}}{2a} \sin\theta\right) \frac{d\theta}{\sin\theta}.
\end{multline}
In what follows, we assume that $m>0$ (the $m=0$ case is similar). 
By \eqref{eq:def-phi} above and Lemma~9 of \cite{DIT:CycleIntegrals}, the right-hand side of \eqref{eq:Q-minus-Q} equals
\[
	\pi \sqrt{\frac{2m}{a}} \, (dD)^\frac14 B(s) \, \chi_D(Q) \, e\pfrac{mb}{2a} J_{s-\frac 12}\pfrac{\pi m \sqrt{dD}}{a}.
\]
Therefore
\begin{multline*}
	\sum_{Q\in \Gamma \backslash \sQ_{dD}} \chi_D(Q) \int_{C_Q} G_{m,Q}(\tau,s) \, d\tau_Q\\
	= \pi \sqrt{2m} (dD)^\frac14 B(s) \sum_{\substack{Q\in \Gamma_\infty \backslash \sQ_{dD} \\ a>0}} \frac{\chi_D(Q)}{\sqrt{a}} e\pfrac{mb}{2a} J_{s-\frac12} \pfrac{\pi m \sqrt{dD}}{a}.
\end{multline*}
Let $\sQ_{dD}^+ = \{[a,b,c]\in \sQ_{dD} : a>0\}$. 
Since $\pmatrix1k01[a,b,c]=[a,b-2ka,*]$, we have a bijection
\[
	[a,b,c] \longleftrightarrow (a,b\bmod 2a)
\]
between $\Gamma_\infty \backslash \sQ_{dD}^+$ and $\{(a,b) : a\in \N \text{ and } 0\leq b<2a\}$. 
Therefore,
\begin{multline*}
	\sum_{Q\in \Gamma \backslash \sQ_{dD}} \chi_D(Q) \int_{C_Q} G_{m,Q}(\tau,s) \, d\tau_Q\\
	= \pi \sqrt{2m} (dD)^\frac14 B(s) \sum_{a=1}^\infty a^{-\frac12} J_{s-\frac12} \pfrac{\pi m \sqrt{dD}}{a} \sum_{\substack{b(2a) \\ \frac{b^2-dD}{4a}\in \Z}} \chi\left([a,b,\tfrac{b^2-dD}{4a}]\right) e\pfrac{mb}{2a}.
\end{multline*}
The latter sum is equal to $\frac12 S_m(d,D,4a)$, so we conclude (after replacing $a$ by $c$) that
\[
	\sum_{Q\in \Gamma \backslash \sQ_{dD}} \frac{\chi_D(Q)}{B(s)} \int_{C_Q} G_{m,Q}(\tau,s) \, d\tau_Q = 
	\tfrac{\pi}{\sqrt 2} m^{\frac 12} (dD)^{\frac 14} \sum_{c=1}^\infty \frac{S_m(d,D;4c)}{c^{\frac 12}} J_{s-\frac 12} \left( \frac{\pi m\sqrt{dD}}{c} \right). \qedhere
\]
\end{proof}

We now complete the proof of Theorem \ref{thm:divisor-sum}, following the proof of Theorem~3 in \cite{DIT:CycleIntegrals}. 
Let
\[
	T_m(s) := \sum_{Q\in \Gamma \backslash \sQ_{dD}} \frac{\chi_D(Q)}{B(s)} \int_{C_Q} G_{m,Q}(\tau,s) d\tau_Q.
\] 
Recall that $d\tau_Q=\sqrt{dD} \, d\tau/Q(\tau,1)$. By \eqref{def-j-m-Q-tau-s} and \eqref{eq:d-D-symmetric}, to prove Theorem~\ref{thm:divisor-sum} we need to show that
\begin{equation} \label{to-show}
	\sum_{n|m} \pfrac Dn (m/n) \, a\left(d,\frac{m^2 D}{n^2}\right) = (dD)^{-\frac 12}\lim_{s\to 1}\left( T_m(s) - \frac{2\pi^{s+\frac 12}m^{1-s}\sigma_{2s-1}(m)}{\Gamma(s+\frac 12)\zeta(2s-1)}\,T_0(s) \right).
\end{equation}
By Proposition 3 of \cite{DIT:CycleIntegrals} we have
\[
	S_m(d,D;4c) = \tfrac12 \sum_{n|(m,c)} \ptfrac Dn \sqrt{\tfrac nc} \, K^+\left(d,\tfrac{m^2D}{n^2};\tfrac{4c}{n}\right),
\]
which, together with Proposition \ref{prop:tr-g-m}, gives
\begin{multline} \label{eq:kloosterman-divisor}
	T_m(s) =
	\begin{dcases}
		\tfrac{\pi}{2\sqrt2}m^\frac12 (dD)^\frac 14 \sum_{n|m} \ptfrac Dn n^{-\frac12} \sum_{c=1}^\infty c^{-1} K^+\left(d,\tfrac{m^2}{n^2}D;4c\right) J_{s-\frac12}\left(\tfrac{\pi m \sqrt{dD}}{nc}\right) &\text{ if }m>0,\\
		2^{-s-2} (dD)^{\frac s2} \sum_{n=1}^\infty \ptfrac Dn n^{-s} \sum_{c=1}^\infty c^{-s-\frac12} K^+(d,0;4c)&\text{ if }m=0.
	\end{dcases}
\end{multline}
Comparing \eqref{eq:kloosterman-divisor} with \eqref{eq:b-d-D-s-formula}, we see that
\begin{equation} \label{eq:T-m-b-d-D}
	T_m(s) = 
	\begin{dcases}
		\sum_{n|m} \ptfrac Dn \, b( d,\tfrac{m^2}{n^2}D,\tfrac s2+\tfrac 14)&\text{ if }m>0,\\
		\pi^{-\frac{s+1}{2}}2^{s-1}D^\frac s2 L_D(s) \, b(d,0,\tfrac s2+\tfrac 14)&\text{ if }m=0,
	\end{dcases}
\end{equation}
where $L_D(s) = \sum_{n>0}\pfrac Dn n^{-s}$ is the Dirichlet $L$-function.
By \eqref{eq:a-d-D-limit} and \eqref{eq:T-m-b-d-D}, the left-hand side of \eqref{to-show} equals
\[
	(dD)^{-\frac 12}\lim_{s\to 1} \left(T_m(s) - \frac{2^{1-s}\pi^{\frac{s+1}2}D^{-\frac s2}}{L_D(s)b(0,0,\tfrac s2+\tfrac 14)} T_0(s) \sum_{n|m} \ptfrac Dn b(0,\tfrac{m^2D}{n^2},\tfrac s2+\tfrac 14)\right).
\]
It remains to show that
\[
	b(0,0,\tfrac s2+\tfrac 14)^{-1} \sum_{n|m} \ptfrac Dn b(0,\tfrac{m^2D}{n^2},\tfrac s2+\tfrac 14) = \frac{2^s D^\frac s2 m^{1-s}\sigma_{2s-1}(m)\pi^\frac s2 L_D(s)}{\Gamma(s+\tfrac 12)\zeta(2s-1)},
\]
which follows from Lemma 4 of \cite{DIT:CycleIntegrals}. \qed

\bibliographystyle{plain}
\bibliography{cycle_integrals_j_bib}

\end{document}